\documentclass[10pt]{article}

\usepackage{hyperref}
\usepackage{xcolor}
\usepackage{amssymb}
\usepackage{amsfonts}
\usepackage{amsthm}
\usepackage{amsmath}
\usepackage{float}
\usepackage{bigints}
\usepackage{fullpage}
\usepackage{mathtools}
\usepackage[utf8]{inputenc} 
\usepackage{cancel}
\usepackage{verbatim}

\newtheorem{theorem}{Theorem}

\newtheorem{lemma}[theorem]{Lemma}

\newtheorem{proposition}[theorem]{Proposition}
\newtheorem{conjecture}[theorem]{Conjecture}
\newtheorem{definition}{Definition}

\newtheorem{corollary}[theorem]{Corollary}
\newtheorem{remark}[theorem]{Remark}

\newcommand{\x}{\mathbf{x}}
\newcommand{\p}{\mathbf{p}}
\newcommand{\q}{\mathbf{q}}
\newcommand{\y}{\mathbf{y}}
\newcommand{\z}{\mathbf{z}}
\newcommand{\oo}{\mathbf{o}}
\newcommand{\aaa}{\mathbf{a}}
\newcommand{\bbb}{\mathbf{b}}
\newcommand{\ccc}{\mathbf{c}}

\newcommand{\B}{\mathbf{B}}
\newcommand{\C}{\mathbf{C}}
\newcommand{\A}{\mathbf{A}}
\newcommand{\K}{\mathbf{K}}

\newcommand{\R}{\mathbf{R}_{+}}

\newcommand{\Ee}{\mathbb{E}}

\newcommand{\Ss}{\mathbb{S}}
\newcommand{\Mm}{\mathbb{M}}

\newcommand{\dist}{{\rm dist}}

\newcommand{\Ed}{\Ee^d}

\newcommand{\noshow}[1]{}
\newcommand{\Sedm}{{\mathbb S}^{d-1}}

\title{Volumetric bounds for intersections of congruent balls
\footnote{\hspace*{.35cm} 2010 Mathematics Subject Classification: 52A20, 52A22.}}

\author{K\'{a}roly Bezdek\thanks{Partially supported by a Natural Sciences and 
Engineering Research Council of Canada Discovery Grant.}}

\date{}

\begin{document}

\maketitle

\begin{abstract}
\noindent We investigate the intersections of balls of radius $r$, called $r$-ball bodies, in Euclidean $d$-space. An $r$-lense (resp., $r$-spindle) is the intersection of two balls of radius $r$ (resp., balls of radius $r$ containing a given pair of points). We prove that among $r$-ball bodies of given volume, the $r$-lense (resp., $r$-spindle) has the smallest inradius (resp., largest circumradius). In general, we upper (resp., lower) bound the intrinsic volumes of $r$-ball bodies of given inradius (resp., circumradius). This complements and extends some earlier results on volumetric estimates for $r$-ball bodies. 
\end{abstract}

\section{Introduction}\label{sec:intro}

Let $\Ee^d$ denote the $d$-dimensional Euclidean vector space, with inner product $\langle\cdot ,\cdot\rangle$ and norm $\|\cdot\|$. Its unit sphere centered at the origin $\oo$ is $\Sedm:= \{\x\in\Ee^d\ |\ \|\x\|= 1\}$. The closed Euclidean ball of radius $r$ centered at $\p\in\Ed$ is denoted by $\B^d[\p,r]:=\{\q\in\Ed\ |\  |\p-\q|\leq r\}$.  Lebesgue measure on $\Ee^d$ is denoted by $V_d(\cdot)$ and spherical Lebesgue measure on $\Sedm$ by $SV_{d-1}(\cdot )$. If $A\subset\Ed$ is a compact convex set, and $0\leq k<d$, then we denote the {\it $k$th intrinsic volume} of $A$ by $V_k(A)$, which can be defined via the Steiner formula:

\begin{equation}\label{Steiner}
V_d(A+\epsilon  \B^d[\oo, 1])=\sum_{i=1}^d\omega_{d-i}V_i(A){\epsilon}^{d-i}.
\end{equation} 
Here $V_d(A)$ (resp.,$V_d(A+\epsilon  \B^d[\oo, 1])$) is called the {\it volume} of $A$ (resp., $A+\epsilon  \B^d[\oo, 1])$), $2V_{d-1}(A)$ is the {\it surface area} of $A$, $\frac{2\omega_{d-1}}{d\omega_d}V_1(A)$ is equal to the {\it mean width} of $A$, and $V_0(A)=1$, where $\omega_d$ stands for the volume of a $d$-dimensional unit ball, i.e., $\omega_d:=\frac{\pi^{\frac{d}{2}}}{\Gamma(1+\frac{d}{2})}$.

\begin{definition}\label{r-ball body}
For a set $\emptyset\neq X\subseteq\Ee^d$, and $r>0$ let the {\rm $r$-ball body} $X^r$ generated by $X$ be defined by $X^r:=\bigcap_{\x\in X}\B^d[\x, r]$. If $X\subset \Ee^d$ is a finite set, then we call $X^r$ the $r$-ball polyhedron generated by $X$ in $\Ee^d$.
\end{definition}

We note that $r$-ball bodies and $r$-ball polyhedra have been intensively studied (under various names) from the point of view of convex and discrete geometry in a number of publications (see the recent papers \cite{BLNP}, \cite{JMR}, \cite{KMP}, \cite{LNT}, \cite{MRS}, and the references mentioned there). In particular, the following {\it Blaschke--Santal\'o-type inequalities} have been proved by Paouris and Pivovarov (Theorem 3.1 in \cite{PaPi}) as well as the author (Theorem 1 in \cite{Be19}) for $r$-ball bodies in $\Ee^d$. Let $\A\subset\Ee^d$, $d>1$ be a compact set of volume $V_{d}(\A)>0$ and $r>0$. If $\B\subset\Ee^d$ is a ball with $V_{d}(\A)=V_{d}(\B)$, then 
\begin{equation}\label{Bezdek-inequalities}
V_{k}(\A^r)\leq V_{k}(\B^r)
\end{equation}
holds for all $1\le k\le d$. In order to state an extension of (\ref{Bezdek-inequalities}) to non-Euclidean spaces we recall the following. Let $\Mm^d$, $d>1$ denote the $d$-dimensional Euclidean, hyperbolic, or spherical space, i.e., one of the simply connected complete Riemannian manifolds of constant sectional curvature. Since simply connected complete space forms, the sectional curvature of which have the same sign are similar, we may assume without loss of generality that the sectional curvature $\kappa$ of $\Mm^d$ is $0, -1$, or $1$. Let $\R$ denote the set of positive real numbers for $\kappa\leq 0$ and the half-open interval $(0, \frac{\pi}{2}]$ for $\kappa =1$. Let $\dist_{\Mm^d}(\x,\y)$ stand for the {\it geodesic distance} between the points $\x\in{\Mm^d}$ and $\y\in{\Mm^d}$. Furthermore, let $\B_{\Mm^d}[\x, r]$ denote the closed $d$-dimensional ball with center $\x\in\Mm^d$ and radius $r\in\R$ in $\Mm^d$, i.e., let $\B_{\Mm^d}[\x, r]:=\{\y\in\Mm^d\ |\dist_{\Mm^d}(\x,\y)\leq r\}$. Finally, for a set $X\subseteq\Mm^d$, $d>1$ and $r\in\R$ let the {\it $r$-ball body} $X^r$ generated by $X$ be defined by $X^r:=\bigcap_{\x\in X}\B_{\Mm^d}[\x, r]$. The following extension of (\ref{Bezdek-inequalities}) to $\Mm^d$ has been proved by the author in \cite{Be18}. Let $\A\subseteq\Mm^d$, $d>1$ be a compact set of volume $V_{\Mm^d}(\A)>0$ and $r\in\R$. If $\B\subseteq\Mm^d$ is a ball with $V_{\Mm^d}(\A)=V_{\Mm^d}(\B)$, then 
\begin{equation}\label{Bezdek-volume-inequality}
V_{\Mm^d}(\A^r)\leq V_{\Mm^d}(\B^r).
\end{equation}
We note that somewhat earlier Gao, Hug, and Schneider \cite{GHSch} proved a special case of (\ref{Bezdek-volume-inequality}) namely, when $\Mm^d=\Ss^d$ and $r=\frac{\pi}{2}$. On the other hand, (\ref{Bezdek-inequalities}) and (\ref{Bezdek-volume-inequality}) have been used in  \cite{Be19} and \cite{Be18} to prove the longstanding Kneser--Poulsen conjecture for uniform contractions of sufficiently many congruent balls in $\Mm^d$ (see also Theorem 1.4 and its proof in \cite{BeNa}). Next, we discuss the following related result of the author and Schneider \cite{BeSch09}, which is again on upper bounding the volume of $r$-ball bodies for $r=\frac{\pi}{2}$ in $\Ss^d$. In order to state it, recall that a {\it spherically convex body} is a closed, spherically convex subset $\K$ of $\Ss^d$ with interior points and lying in some closed hemisphere, thus, the intersection of $\Ss^d$ with a $(d+1)$-dimensional closed convex cone of $\Ee^{d+1}$ different from $\Ee^{d+1}$. The {\it inradius} $r_{in}(\K)$ of $\K$ is the angular radius of the largest spherical ball contained in $\K$. Also, recall that a {\it lune} in $\mathbb{S}^d$ is the $d$-dimensional intersection of $\mathbb{S}^d$ with two closed halfspaces of $\mathbb{E}^{d+1}$ with the origin $\mathbf{o}$ in their boundaries. Evidently, the inradius of a lune is half the interior angle between the two defining hyperplanes. Now, the main result of \cite{BeSch09} on {\it volume maximizing lunes} can be stated as follows. For a somewhat simpler and more direct proof by Akopyan and Karasev see  Section 6 in \cite{AkKa12} as well as Section 8.4 in \cite{Be}. If $\mathbf{K}$ is a spherically convex body in $\mathbb{S}^d, d\ge 2$, then
\begin{equation}\label{Bezdek-Schneider-spherical-inequality}
 {\rm Svol}_d(\mathbf{K})\le \frac{(d+1)\omega_{d+1}}{\pi} r_{in}(\mathbf{K}).
 \end{equation}
Equality holds if and only if $\mathbf{K}$ is a lune. For the sake of completeness we note that (\ref{Bezdek-Schneider-spherical-inequality}) is used in \cite{BeSch09} to derive the following spherical version of a Tarski-type theorem of Kadets (\cite{Ka05}). If the spherically convex bodies $\mathbf{K}_1,\dots,\mathbf{K}_n$ cover the spherical ball $\mathbf{B}$ of radius $r_{in}(\mathbf{B}) \ge\frac{\pi}{2}$ in $\mathbb{S}^d, d\ge 2$, then $ \sum_{i=1}^n r_{in}(\mathbf{K}_i)\ge r_{in}(\mathbf{B}).$

The main goal of this note is to extend (\ref{Bezdek-Schneider-spherical-inequality}) to Euclidean spaces as follows. Let $\K\subset\Ee^d$ be a convex body, i.e., let $\K$ be a compact convex set with nonempty interior in $\Ee^d$. Then its {\it inradius} $r_{in}(\K)$ (resp., {\it circumradius} $r_{cr}(\K)$) is the radius of the largest (resp., smallest) ball contained in (resp., containing) $\K$. Furthermore, if $\K$ is an intersection of two balls of radius $r$, then we call it an {\it $r$-lense} of $\Ee^d$. In particular, we are going to use the notation $L_{r, \rho, d}$ for an $r$-lense whose inradius is $\rho$ in $\Ee^d$, where $r\geq\rho>0$.

\begin{theorem}\label{max-volume-lense}
Let $r>r_0>0$, $N>1, d>1$, and let $P:=\{\p_1,\p_2,\dots ,\p_N\}\subset\Ee^d$ with $r_{cr}(P)=r_0$. Then
\begin{equation}\label{Bezdek-1}
V_d(P^r)\leq V_d(L_{r, r-r_0, d}).
\end{equation}
\end{theorem} 

\begin{remark}\label{equivalent-lense}
We note that $r_{in}(P^r)=r-r_0$ in Theorem~\ref{max-volume-lense}. Thus, it follows that Theorem~\ref{max-volume-lense} is equivalent to the statement that among $r$-ball polyhedra (resp., $r$-ball bodies) of given volume in $\Ee^d$, the $r$-lense has the smallest inradius.
\end{remark}

One can derive from Theorem~\ref{max-volume-lense} the following weaker version of Kadets's theorem. (It is worth emphasizing that our proof of Corollary~\ref{Kadets-type} is volumetric while the proof of Kadets's theorem published in \cite{Ka05} is not.)

\begin{corollary}\label{Kadets-type}
Let $\B$ be a ball of radius $r>0$ in $\Ee^d, d>1$ and let $\C_i$ be an $r_i$-ball body with $r_i\leq r$ for $1\leq i\leq n$ in $\Ee^d$ such that $\B\subseteq\bigcup_{i=1}^{n}\C_i$.
Then 
\begin{equation}\label{Kadets-type-inequality} 
r\leq\sum_{i=1}^{n}r_{in}(\C_i\cap\B).
\end{equation}
\end{corollary}

\begin{theorem}\label{upper-bound-for-intrinsic-volume}
Let $r>r_0>0$, $N>1, d>k>0$, and let $P:=\{\p_1,\p_2,\dots ,\p_N\}\subset\Ee^d$ with $r_{cr}(P)=r_0$. Then
\begin{equation}\label{Bezdek-2}
V_k(P^r)\leq V_k\left(L_{r, r-\sqrt{\frac{d+1}{2d}}r_0, d}\right).
\end{equation}
\end{theorem}

In connection with Theorems~\ref{max-volume-lense} and~\ref{upper-bound-for-intrinsic-volume} it is natural to raise

\begin{conjecture}\label{max-intrinsic-volume-lense}
Let $r>r_0>0$, $N>1, d>k>0$, and let $P:=\{\p_1,\p_2,\dots ,\p_N\}\subset\Ee^d$ with $r_{cr}(P)=r_0$. Then
\begin{equation}\label{Bezdek-3}
V_k(P^r)\leq V_k\left(L_{r, r-r_0, d}\right).
\end{equation}
\end{conjecture}

\begin{remark}\label{planar case}
Recall that according to \cite{BoDr} (see also \cite{FoKuVi}) the $r$-lense has maximal perimeter among $r$-ball bodies of equal area in $\Ee^2$. This statement and Theorem~\ref{max-volume-lense} imply Conjecture~\ref{max-intrinsic-volume-lense} for $d=2$ and $k=1$. Hence, if $r>r_0>0$, $N>1$, $d=2$, $k=1$, and $P:=\{\p_1,\p_2,\dots ,\p_N\}\subset\Ee^2$ with $r_{cr}(P)=r_0$, then 
\begin{equation}
V_1\left(P^r\right) \leq V_1\left(L_{r, r-r_0, 2}\right).
\end{equation}\end{remark}

\begin{definition}\label{r-convex hull}
Let $\emptyset\neq K\subset\Ee^d$, $d>1$ and $r>0$. Then the {\rm $r$-ball convex hull} ${\rm conv}_rK$ of $K$ is defined by ${\rm conv}_rK:=\bigcap\{ \B^d[\x, r]\ |\ K\subseteq \B^d[\x, r]\}$. Moreover, let the $r$-ball convex hull of $\Ee^d$ be $\Ee^d$. Furthermore, we say that  $K\subseteq\Ee^d$ is {\rm $r$-ball convex} if $K={\rm conv}_rK$. 
\end{definition}

We note that clearly, ${\rm conv}_rK=\emptyset$ if and only if $K^r=\emptyset$. Moreover, $\emptyset\neq K\subset\Ee^d$ is $r$-ball convex if and only if $K$ is an $r$-ball body. If $K:=\{\x,\y\}\subset\Ee^d$ with $0<\|\x-\y\|\leq 2r$, then ${\rm conv}_rK$ is called an {\it $r$-spindle} with $r_{cr}=\frac{1}{2}\|\x-\y\|$. In particular, we are going to use the notation $S_{r, \lambda, d}$ for an $r$-spindle whose circumradius is $\lambda$ in $\Ee^d$, where $r\geq\lambda>0$.

\begin{theorem}\label{min-volume-spindle}
Let $r>r_0>0$, $N>1, d>1$, and let $P:=\{\p_1,\p_2,\dots ,\p_N\}\subset\Ee^d$ with $r_{cr}(P)=r_0$. Then
\begin{equation}\label{Bezdek-4}
V_d(S_{r, r_0, d})\leq V_d\left({\rm conv}_rP\right).
\end{equation}
\end{theorem} 

\begin{remark}\label{equivalent-spindle}
Clearly, Theorem~\ref{min-volume-spindle} is equivalent to the statement that among $r$-ball bodies of given volume in $\Ee^d$, the $r$-spindle has the largest circumradius.
\end{remark}

\begin{corollary}\label{lower-bound-for-intrinsic-volume}
Let $r>r_0>0$, $N>1, d>k>0$, and let $P:=\{\p_1,\p_2,\dots ,\p_N\}\subset\Ee^d$ with $r_{cr}(P)=r_0$. Then
\begin{equation}\label{Bezdek-6}
\frac{\binom{d}{k}\omega_d^{1-\frac{k}{d}}}{\omega_{d-k}}\left(V_d(S_{r, r_0, d})\right)^{\frac{k}{d}}\leq V_k\left({\rm conv}_rP\right).
\end{equation}
Moreover, if $r>r_0>0$, $N>1$, $d=2$, $k=1$, and $P:=\{\p_1,\p_2,\dots ,\p_N\}\subset\Ee^2$ with $r_{cr}(P)=r_0$, then 
\begin{equation}\label{Bezdek-5}
V_1\left(S_{r, r_0, 2}\right)\leq V_1\left({\rm conv}_rP\right).
\end{equation}
\end{corollary}

We conclude this section by raising

\begin{conjecture}\label{min-intrinsic-volume-spindle}
Let $r>r_0>0$, $N>1, d>k>0$, and let $P:=\{\p_1,\p_2,\dots ,\p_N\}\subset\Ee^d$ with $r_{cr}(P)=r_0$. Then
\begin{equation}\label{Bezdek-7}
V_k(S_{r, r_0, d})\leq V_k\left({\rm conv}_rP\right).
\end{equation}
\end{conjecture}

\begin{remark}\label{extending-Linhart}
Conjecture~\ref{min-intrinsic-volume-spindle} for $k=1$ states that among $r$-ball bodies of given circumradius the $r$-spindle possesses the smallest
mean width. If true, then this result could be regarded as an extension of the relevant inequality of Linhart (see inequality (1) in \cite{Li}) from convexity to $r$-convexity.
\end{remark}

In the rest of the paper we prove the theorems stated.

\section{Proof of Theorem~\ref{max-volume-lense}}

\begin{lemma}\label{max-cr}
Let $r>r_0>0$, $N>1, d>1$, and let $P:=\{\p_1,\p_2,\dots ,\p_N\}\subset\B^d[\oo, r_0]$ with $r_{cr}(P)=r_0$. Then
\begin{equation}\label{Bezdek-8}
P^r\subset\B^d\left[\oo, \sqrt{r^2-r_0^2}\right]
\end{equation}
and so, $r_{cr}(P^r)\leq r_{cr}(L_{r, r-r_0, d})=\sqrt{r^2-r_0^2}$.
\end{lemma}

\begin{proof}
First, recall that Lemma 5 of \cite{Be18} and (ii) of Corollary 3.4 of \cite{BLNP} imply 
\begin{equation}\label{Bezdek-9a}
P^r=({\rm conv}_rP)^r  \ {\rm and}\ \left(P^r\right)^r={\rm conv}_rP
\end{equation}
from which it follows in a straightforward way that
\begin{equation}\label{Bezdek-9b}
r_{in}^{\oo}\left({\rm conv}_rP\right)+r_{cr}^{\oo}(P^r)\leq r,
\end{equation}
where $r_{in}^{\oo}\left({\rm conv}_rP\right):=\max\{\rho\ |\ B^d[\oo,\rho]\subset {\rm conv}_rP\}$ and $r_{cr}^{\oo}\left(P^r\right):=\min\{\lambda\ |\ P^r\subset B^d[\oo,\lambda]\}$. 
Thus, (\ref{Bezdek-9b}) implies that in order to prove (\ref{Bezdek-8}) it is sufficient to show 
\begin{equation}\label{Bezdek-10}
r_{in}\left(S_{r, r_0, d}\right)=r-\sqrt{r^2-r_0^2}\leq  r_{in}^{\oo}\left({\rm conv}_rP\right),
\end{equation}
where $S_{r, r_0, d}$ is an $r$-spindle with circumradius $r_0$. Next, without loss of generality, we may assume
that the circumscribed ball of $S_{r, r_0, d}$ is $\B^d[\oo, r_0]$ and the inscribed ball of $S_{r, r_0, d}$
is $\B^d\left[\oo, r-\sqrt{r^2-r_0^2}\right]$. As $\B^d[\oo, r_0]$ is the smallest ball containing the {\it convex hull} ${\rm conv}P$ of $P$ (resp., ${\rm conv}_rP$), therefore there must exist a simplex $\Delta$ of dimension $l$ ($1\leq l\leq d$) spanned by $l+1$ points of $P$ lying on $r_0\Ss^{d-1}={\rm bd}(\B^d[\oo, r_0])$ such that $\oo\in{\rm relint}(\Delta)$, where ${\rm bd}(\cdot)$ (resp., ${\rm relint}(\cdot)$) refers to the {\it boundary} (resp., {\it relative interior}) of the corresponding set in $\Ee^d$. (Clearly, the circumscribed ball of $\Delta$ (resp., ${\rm conv}_r\Delta$) is $\B^d[\oo, r_0]$.) Without loss of generality, we may assume that $\Delta={\rm conv}\{\p_1,\p_2,\dots ,\p_{l+1}\}$ with $\{\p_1,\p_2,\dots ,\p_{l+1}\}\subset r_0\Ss^{d-1}$. As ${\rm conv}_r\Delta\subseteq {\rm conv}_rP$ therefore if 
\begin{equation}\label{Bezdek-11}
\B^d\left[\oo, r-\sqrt{r^2-r_0^2}\right]\subseteq{\rm conv}_r\Delta
\end{equation} 
holds, then (\ref{Bezdek-10}) follows. So, we are left to show that indeed, (\ref{Bezdek-11}) holds. In order to see this, recall Lemma 3.1 and Corollary 3.4 of \cite{BLNP} according to which for each boundary point $\p$ of ${\rm conv}_r\Delta$ there exists a $(d-1)$-dimensional sphere $S$ of radius $r$ (called {\it supporting $r$-sphere of ${\rm conv}_r\Delta$}) that bounds a ball $\B$ (called {\it supporting $r$-ball of ${\rm conv}_r\Delta$}) in $\Ee^d$ such that ${\rm conv}_r\Delta\subseteq\B$ and $\p\in S\cap {\rm conv}_r\Delta$. Moreover, ${\rm conv}_r\Delta$ is the intersection of its supporting $r$-balls. Thus, (\ref{Bezdek-11}) follows if one can prove that 
\begin{equation}\label{Bezdek-12}
\B^d\left[\oo, r-\sqrt{r^2-r_0^2}\right]\subseteq\B
\end{equation}
holds for any supporting $r$-ball $\B$ of ${\rm conv}_r\Delta$. Finally, we prove (\ref{Bezdek-12}) as follows. First, we note that clearly, $\p_i\in S={\rm bd}(\B)$ for some $1\leq i\leq l+1$. Moreover, $\{\p_i,-\p_i\}\subset r_0\Ss^{d-1}$ and ${\rm conv}_r\{\p_i,-\p_i\}$ is an $r$-spindle of inradius $r-\sqrt{r^2-r_0^2}$. Hence, if $-\p_i\in\B$, then one obtains (\ref{Bezdek-12}) in a straightforward way. So, the case left is when $-\p_i\notin\B$. But then, $\B\cap r_0\Ss^{d-1}$ is a spherical cap of angular radius $<\frac{\pi}{2}$ on 
$ r_0\Ss^{d-1}$ containing $\{\p_1,\p_2,\dots ,\p_{l+1}\}$ and clearly implying that $\oo\notin  {\rm relint}(\Delta)$, a contradiction. This completes the proof of Lemma~\ref{max-cr}.
\end{proof}

For the purpose of the next statement recall that $\mathbf{B}_{\Ss^d}[\mathbf{x}, \epsilon]$ denotes the closed ball of angular radius $\epsilon\le\frac{\pi}{2}$ centered at the point $\mathbf{x}$ in $\mathbb{S}^d$. Furthermore, for any subset ${X}$ of $\mathbb{S}^d$ let ${X}_{\epsilon}:=\cup_{\mathbf{x}\in {X}}\mathbf{B}_{\Ss^d}[\mathbf{x}, \epsilon]$ be called the {\it $\epsilon$-neighbourhood} of ${X}$ in $\mathbb{S}^d$. The following statement, which we need for the proof of Theorem~\ref{max-volume-lense}, has been proved by Akopyan and Karasev (see Lemma 7 in \cite{AkKa12} as well as Lemma 8.4.3 in \cite{Be}). In what follows, we reprove it in a similar but simpler way.

\begin{lemma}\label{spherical-neighbourhood-sets}
Let ${\mathbf{X}}$ be a closed subset of $\mathbb{S}^d$ not lying on an open hemisphere of $\mathbb{S}^d$. Then for any $\epsilon\le\frac{\pi}{2}$ the inequality 
\begin{equation}\label{Bezdek-13}
{SV}_d({\mathbf{X}}_{\epsilon})\geq{SV}_d(\hat{{\mathbf{X}}}_{\epsilon})
\end{equation}
holds, where $\hat{{\mathbf{X}}}$ is a pair of antipodal points of $\mathbb{S}^d$.
\end{lemma}
\begin{proof}
It is sufficient to prove Lemma~\ref{spherical-neighbourhood-sets} for finite $\mathbf{X}$ say, for $\mathbf{X}:=\{\mathbf{x}_1, \mathbf{x}_2, \dots , \mathbf{x}_m\}\subset\mathbb{S}^d$. Take the (nearest point) Voronoi tiling of $\mathbb{S}^d$ generated by $\mathbf{X}$ with $\mathbf{V}_i$ standing for the Voronoi cell assigned to the point $\mathbf{x}_i$, $1\le i\le m$. Let $\mathbf{H}_i$ be the closed hemisphere of $\mathbb{S}^d$ centered at $\mathbf{x}_i$, $1\le i\le m$. As, by assumption, $\mathbf{X}$ does not lie on an open hemisphere of $\mathbb{S}^d$ therefore, $\mathbf{V}_i\subseteq \mathbf{H}_i$ holds for all $1\le i\le m$. The following lower bound for the {\it density} $\frac{SV_{d}(\B_{\Ss^d}[\x_i,\epsilon]\cap\mathbf{V}_i)}{SV_d(\mathbf{V}_i)}$ of $\B_{\Ss^d}[\x_i,\epsilon]\cap\mathbf{V}_i$ within $\mathbf{V}_i$, is the core part of our proof of (\ref{Bezdek-13}).
\begin{proposition}\label{density-lower-bound}
\begin{equation}\label{Bezdek-14}
\frac{SV_{d}(\B_{\Ss^d}[\x_i,\epsilon]\cap\mathbf{V}_i)}{SV_d(\mathbf{V}_i)}\geq \frac{SV_{d}(\B_{\Ss^d}[\x_i,\epsilon])}{SV_d(\mathbf{H}_i)}
\end{equation}
holds for all $1\leq i\leq m$.
\end{proposition}
\begin{proof}
For any $\x,\y\in\Ss^d$ with $\x\neq -\y$ let $[\x,\y]_{\Ss^d}$ denote the {\it geodesic segment} connecting $\x$ and $\y$, i.e., let $[\x,\y]_{\Ss^d}$ stand for the shorter closed unit circle arc connecting $\x$ and $\y$ in $\Ss^d$.
\begin{definition}\label{starshaped-sets}
For $\aaa\in {\rm bd}(\mathbf{H}_i)$, $1\leq i\leq m$ let $\bbb:=[\aaa,\x_i]_{\Ss^d}\cap{\rm bd}(\B_{\Ss^d}[\x_i,\epsilon])$ and $\ccc:=[\aaa,\x_i]_{\Ss^d}\cap {\rm bd}(\mathbf{V}_i)$, where ${\rm bd}(\cdot)$ refers to the {\it boundary} of the corresponding set in $\Ss^d$. Then let $\A_i:=\bigcup\{[\aaa,\x_i]_{\Ss^d}|\aaa\in {\rm bd}(\mathbf{H}_i)\ {\rm with}\ \ccc\in[\bbb ,\x_i]\}$. Moreover, let $\A'_i:={\rm bd}(\mathbf{H}_i)\setminus\A_i$.
\end{definition}

Clearly, for any $\aaa\in\A'_i$, $1\leq i\leq m$ we have $\bbb\in{\rm relint}([\ccc,\x_i])$, where ${\rm relint}(\cdot)$ denotes the {\it relative interior} of the corresponding set in $\Ss^d$. Moreover, we note that $\A_i$ as well as $\A'_i$ are starshaped sets with respect to $\x_i$ in $\Ss^d$, where $1\leq i\leq m$. Thus, it follows in a rather straightforward way that
$$
\frac{SV_{d}(\B_{\Ss^d}[\x_i,\epsilon]\cap\mathbf{V}_i)}{SV_d(\mathbf{V}_i)}=\frac{SV_{d}(\A_i\cap\mathbf{V}_i)+SV_{d}(\A'_i\cap\B_{\Ss^d}[\x_i,\epsilon])}{SV_d(\mathbf{V}_i)}
$$
$$
=\frac{SV_{d}(\A_i\cap\mathbf{V}_i)+SV_{d}(\A'_i\cap\mathbf{V}_i)\frac{SV_{d}(\A'_i\cap\B_{\Ss^d}[\x_i,\epsilon])}{SV_{d}(\A'_i\cap\mathbf{V}_i)}}{SV_d(\mathbf{V}_i)}\geq
\left(\frac{SV_{d}(\A_i\cap\mathbf{V}_i)+SV_{d}(\A'_i\cap\mathbf{V}_i)}{SV_d(\mathbf{V}_i)}\right)\frac{SV_{d}(\B_{\Ss^d}[\x_i,\epsilon])}{SV_d(\mathbf{H}_i)}
$$
$$
=\frac{SV_{d}(\B_{\Ss^d}[\x_i,\epsilon])}{SV_d(\mathbf{H}_i)},
$$
finishing the proof of Proposition~\ref{density-lower-bound}.
\end{proof}

Thus, Proposition~\ref{density-lower-bound} yields that
$${SV}_d(\mathbf{B}_{\Ss^d}[\mathbf{x}_i, \epsilon]\cap\mathbf{V}_i){SV}_d( \mathbf{H}_i)\ge{SV}_d(\mathbf{B}_{\Ss^d}[\mathbf{x}_i, \epsilon]){SV}_d(\mathbf{V}_i),$$
or equivalently
\begin{equation}\label{spherical-voronoi-decomposition-inequality}
{SV}_d( \mathbf{X}_{\epsilon}\cap\mathbf{V}_i)\frac{{SV}_d( \mathbb{S}^d)}{2}\ge{SV}_d(\mathbf{B}_{\Ss^d}[\mathbf{x}_i, \epsilon]){SV}_d(\mathbf{V}_i)
\end{equation}
holds for all $1\le i\le m$. As $\sum_{i=1}^{m}{SV}_d( \mathbf{X}_{\epsilon}\cap\mathbf{V}_i)={SV}_d(\mathbf{X}_{\epsilon})$ and $\sum_{i=1}^{m}{SV}_d(\mathbf{V}_i)={SV}_d( \mathbb{S}^d)$ therefore (\ref{spherical-voronoi-decomposition-inequality}) implies in a straightforward way that
$${SV}_d(\mathbf{X}_{\epsilon})\ge 2{SV}_d(\mathbf{B}_{\Ss^d}[\mathbf{x}, \epsilon])={SV}_d(\hat{{\mathbf{X}}}_{\epsilon})$$
holds for $\hat{{\mathbf{X}}}=\{\x, -\x\}$ with $\mathbf{x}\in\mathbb{S}^d$. This completes the proof of Lemma~\ref{spherical-neighbourhood-sets}. \end{proof}

Now, we turn to the proof of Theorem~\ref{max-volume-lense}. Without loss of generality we may assume that $P=\{\p_1,\p_2,\dots ,\p_N\}\subset\B^d[\oo, r_0]$ with $r_{cr}(P)=r_0$ implying that there exists a simplex of dimension $l$ ($1\leq l\leq d$) spanned by some points of $P$ say, by $Q:=\{\p_1,\p_2,\dots ,\p_{l+1}\}$ lying on $r_0\Ss^{d-1}={\rm bd}(\B^d[\oo, r_0])$ such that $\oo\in{\rm relint}({\rm conv}(Q))$. As $P^r\subseteq Q^r$
and $r_{cr}(P)=r_{cr}(Q)=r_0$ therefore Theorem~\ref{max-volume-lense} follows from the inequality
\begin{equation}\label{Bezdek-15}
V_d(Q^r)\leq V_d(L_{r, r-r_0, d}),
\end{equation}
where the inscribed ball of $Q^r$ as well as $L_{r, r-r_0, d}$ is $\B^d[\oo, r-r_0]$. Clearly, 
\begin{equation}\label{Bezdek-16}
r_{cr}(L_{r, r-r_0, d})=\sqrt{r^2-r_0^2} \ {\rm and} \ L_{r, r-r_0, d}\subset \B^d\left[\oo, \sqrt{r^2-r_0^2}\right].
\end{equation}
Moreover, Lemma~\ref{max-cr} implies that
\begin{equation}\label{Bezdek-17}
\ Q^r\subset \B^d\left[\oo, \sqrt{r^2-r_0^2}\right]\ \left({\rm and \ therefore}\ r_{cr}(Q^r)\leq\sqrt{r^2-r_0^2}\right).
\end{equation}
Thus, (\ref{Bezdek-16}) and (\ref{Bezdek-17}) yield that
\begin{equation}\label{Bezdek-18} 
V_d(Q^r)=\int_{0}^{\sqrt{r^2-r_0^2}}\sigma(x\Ss^{d-1}\cap Q^r)dx\ {\rm and}\ V_d( L_{r, r-r_0, d})=\int_{0}^{\sqrt{r^2-r_0^2}}\sigma(x\Ss^{d-1}\cap  L_{r, r-r_0, d})dx,
\end{equation} 
where $\sigma$ denotes the proper spherical Lebesgue measure on $x\Ss^{d-1}$. Hence, using (\ref{Bezdek-18}) we get that in order to prove (\ref{Bezdek-15}) it is sufficient to show that
\begin{equation}\label{Bezdek-19}
\sigma(x\Ss^{d-1}\cap Q^r)\leq \sigma(x\Ss^{d-1}\cap  L_{r, r-r_0, d})
\end{equation}
holds for all $x$ with $0\leq x\leq \sqrt{r^2-r_0^2}$. Now, (\ref{Bezdek-19}) holds trivially for all $0\leq x\leq r-r_0=r_{in}(Q^r)=r_{in}(L_{r, r-r_0, d})$. So, we are left with the case when $r-r_0<x\leq \sqrt{r^2-r_0^2}$. Next, notice that according to (\ref{Bezdek-16}) the subset $x\Ss^{d-1}\cap  L_{r, r-r_0, d}$ of $x\Ss^{d-1}$ is the complement of the union of a pair of antipodal (open) spherical caps of angular radius $0<\epsilon\leq\frac{\pi}{2}$. On the other hand, the subset $x\Ss^{d-1}\cap Q^r$ of $x\Ss^{d-1}$ is the complement of the union of $l+1$ (open) spherical caps of angular radius $\epsilon$ centered at the points $-\frac{x}{r_0}\p_i$, $1\leq i\leq l+1$, which do not lie on an open hemisphere of $x\Ss^{d-1}$ because $\oo\in{\rm relint}({\rm conv}(Q))$. Thus, Lemma~\ref{spherical-neighbourhood-sets} implies (\ref{Bezdek-19}) in a straightforward way. This completes the proof of Theorem~\ref{max-volume-lense}.

\section{Proof of Corollary~\ref{Kadets-type}}

Clearly, $\C_i\cap\B$ is an $r$-ball body in $\Ee^d$ for all $1\leq i\leq n$. Thus, Theorem~\ref{max-volume-lense} and $\B\subseteq\bigcup_{i=1}^{n}\C_i$ imply that
\begin{equation}\label{Bezdek-20}
V_d(\B)\leq\sum_{i=1}^nV_d(\C_i\cap\B)\leq\sum_{i=1}^{n}V_d\left(L_{r, r_{in}(\C_i\cap\B), d}\right).
\end{equation}
Finally, we note that in order to have (\ref{Bezdek-20}) one must have $r\leq \sum_{i=1}^n r_{in}(\C_i\cap\B)$, finishing the proof of Corollary~\ref{Kadets-type}.

\section{Proof of Theorem~\ref{upper-bound-for-intrinsic-volume}}

As in the proof of Theorem~\ref{max-volume-lense}, we may assume without loss of generality that $P=\{\p_1,\p_2,\dots ,\p_N\}\subset\B^d[\oo, r_0]$ with $r_{cr}(P)=r_0$. It follows that there exists a simplex $\Delta$ of dimension $l$ ($1\leq l\leq d$) spanned by $l+1$ points of $P$ say, by $Q:=\{\p_1,\p_2,\dots ,\p_{l+1}\}$ lying on $r_0\Ss^{d-1}={\rm bd}(\B^d[\oo, r_0])$ such that $\oo\in{\rm relint}(\Delta)$. Clearly, the circumscribed ball of $\Delta={\rm conv}Q$ is $\B^d[\oo, r_0]$ and 
\begin{equation}\label{Bezdek-21}
r_{cr}(P)=r_{cr}(Q)=r_{0},\ r_{in}(P^r)=r_{in}(Q^r)=r-r_0,\ {\rm and}\ P^r\subseteq Q^r .
\end{equation}

\begin{definition}\label{symmetral}
Let $\emptyset\neq X\subseteq\Ee^d$. Then the {\rm central symmetral} (called also {\rm Minkowski symmetral}) $M_{\oo}(X)$ of $X$ is defined by $M_{\oo}(X):=\frac{1}{2}\left(X+(-X)\right)$.
\end{definition}
For properties of central symmetrization we refer the interested reader to \cite{BGG}. In particular, recall that the Brunn-Minkowski inequality for intrinsic volumes (\cite{Ga}) and (\ref{Bezdek-21}) yield
\begin{equation}\label{Bezdek-22}
V_k(P^r)\leq V_k\left(M_{\oo}(P^r)\right)\leq V_k\left(M_{\oo}(Q^r)\right),
\end{equation}
where $0<k\leq d$.
\begin{lemma}\label{symmetral-of-ball-polyhedron}
\begin{equation}\label{Bezdek-23}
M_{\oo}(Q^r)=\left(M_{\oo}(Q)\right)^r.
\end{equation}
\end{lemma}
\begin{proof} Clearly, (\ref{Bezdek-23}) is equivalent to 
\begin{equation}\label{Bezdek-24}
M_{\oo}(Q^r)=\bigcap\left\{\B^d\left[\frac{1}{2}(\p_i-\p_j), r\right]\bigg|1\leq i,j\leq l+1\right\},
\end{equation}
which we prove as follows. Let $\z\in M_{\oo}(Q^r)=\frac{1}{2}\left(Q^r+(-Q^r)\right)$. Then there exist $\x,\y\in Q^r$ such that $\z=\frac{1}{2}(\x-\y)$. It follows that $\x\in \B^d[\p_i,r]$ and $\y\in \B^d[\p_j,r]$ for all $1\leq i,j\leq l+1$ and therefore
\begin{equation}\label{Bezdek-25}
\z=\frac{1}{2}(\x-\y)\in \frac{1}{2}\B^d[\p_i,r]+\frac{1}{2}\B^d[-\p_j,r]=\B^d\left[\frac{1}{2}(\p_i-\p_j),r\right]  
\end{equation}
holds for all  $1\leq i,j\leq l+1$. Clearly, (\ref{Bezdek-25}) yields $M_{\oo}(Q^r)\subseteq \bigcap\left\{\B^d\left[\frac{1}{2}(\p_i-\p_j), r\right]\bigg|1\leq i,j\leq l+1\right\}$. On the other hand, let $\z'\in \bigcap\left\{\B^d\left[\frac{1}{2}(\p_i-\p_j), r\right]\bigg|1\leq i,j\leq l+1\right\}$. Then $\z'\in \B^d\left[\frac{1}{2}(\p_i-\p_j),r\right]= \frac{1}{2}\left(\B^d[\p_i,r]+(-\B^d[\p_j,r])\right)$ holds for all $1\leq i,j\leq l+1$ and therefore $\z'\in \frac{1}{2}\left(Q^r+(-Q^r)\right)$ implying that $ \bigcap\left\{\B^d\left[\frac{1}{2}(\p_i-\p_j), r\right]\bigg|1\leq i,j\leq l+1\right\}\subseteq M_{\oo}(Q^r)$. This completes the proof of Lemma~\ref{symmetral-of-ball-polyhedron}.
\end{proof}

\begin{corollary}\label{lense-assigned-to-symmetral}
Lemma~\ref{symmetral-of-ball-polyhedron} implies that $M_{\oo}(Q^r)$ is an $\oo$-symmetric $r$-ball polyhedron and therefore
it is contained in an $r$-lense of inradius equal to
\begin{equation}\label{Bezdek-26}
r_{in}\left[M_{\oo}(Q^r)\right]=r_{in}\left[\left(M_{\oo}(Q)\right)^r\right]=r_{in}^{\oo}\left[\left(M_{\oo}(Q)\right)^r\right]=r-r_{cr}^{\oo}\left[M_{\oo}(Q)\right]. 
\end{equation}
Hence, 
\begin{equation}\label{Bezdek-27}
V_k\left(M_{\oo}(Q^r)\right)\leq V_k\left(L_{r, r-r_{cr}^{\oo}\left[M_{\oo}(Q)\right], d}\right)
\end{equation}
holds for all $0<k\leq d$.
\end{corollary}

\begin{lemma}\label{lower-bounding-cr}
Let $Q=\{\p_1,\p_2,\dots ,\p_{l+1}\}\subset r_0\Ss^{d-1}={\rm bd}(\B^d[\oo, r_0])$ be given such that ${\rm conv}Q$ is an $l$-dimensional simplex with $\oo\in{\rm relint}({\rm conv}Q)$ in $\Ee^d$, where $1\leq l\leq d$. Then
\begin{equation}\label{Bezdek-28}
\sqrt{\frac{d+1}{2d}}r_0\leq \sqrt{\frac{l+1}{2l}}r_0\leq r_{cr}^{\oo}\left[M_{\oo}(Q)\right]. 
\end{equation}
\end{lemma}
\begin{proof}
In fact, one may assume that $Q=\{\p_1,\p_2,\dots ,\p_{l+1}\}\subset r_0\Ss^{l-1}\subset\Ee^l$ and ${\rm conv}Q$ is an $l$-dimensional simplex with the origin lying in its interior in $\Ee^l$ (i.e., $\oo\in{\rm int}({\rm conv}Q)$).  
Clearly, $r_{cr}(Q)=r_{cr}^{\oo}(Q)=r_0$ and
\begin{equation}\label{Bezdek-29}
r_{cr}^{\oo}\left[M_{\oo}(Q)\right]=\max\left\{\frac{1}{2}\|\p_i-\p_j\|\bigg| 1\leq i,j\leq l+1\right\}=:\frac{1}{2}{\rm diam}(Q),
\end{equation}
where $M_{\oo}(Q)$ stands for the central symmetral of $Q$ in $\Ee^l\subseteq\Ee^d$ and ${\rm diam}(Q)$ denotes the {\it diameter} of $Q$. Next, recall Jung's theorem
stated as follows (see Theorem 1 in \cite{De}): Let $C\subset\Ee^l$ a compact set having unit circumradius. Then $2\sqrt{\frac{l+1}{2l}}\leq {\rm diam}(C)$. Finally, this theorem of Jung and (\ref{Bezdek-29}) imply in a straightforward way that $\sqrt{\frac{l+1}{2l}}r_0\leq \frac{1}{2}{\rm diam}(Q)=r_{cr}^{\oo}\left[M_{\oo}(Q)\right]$ and so, (\ref{Bezdek-28}) follows. This completes the proof of Lemma~\ref{lower-bounding-cr}.
\end{proof}

Hence, (\ref{Bezdek-22}), (\ref{Bezdek-27}), and (\ref{Bezdek-28}) yield (\ref{Bezdek-2}), finishing the proof of Theorem~\ref{upper-bound-for-intrinsic-volume}.

\section{Proof of Theorem~\ref{min-volume-spindle}}

As in the proof of Theorem~\ref{max-volume-lense}, we may assume without loss of generality that $P=\{\p_1,\p_2,\dots ,\p_N\}\subset\B^d[\oo, r_0]$ with $r_{cr}(P)=r_0$. It follows that there exists a simplex $\Delta$ of dimension $l$ ($1\leq l\leq d$) spanned by $l+1$ points of $P$ say, by $Q:=\{\p_1,\p_2,\dots ,\p_{l+1}\}$ lying on $r_0\Ss^{d-1}={\rm bd}(\B^d[\oo, r_0])$ such that $\oo\in{\rm relint}(\Delta)$. Clearly, the circumscribed ball of ${\rm conv}_r\Delta={\rm conv}_rQ$ is $\B^d[\oo, r_0]$ and so,
\begin{equation}\label{Bezdek-30}
r_{cr}^{\oo}\left({\rm conv}_rQ\right)=r_{0}\ {\rm with}\ {\rm conv}_rQ\subseteq{\rm conv}_rP\ {\rm implying}\ V_d\left({\rm conv}_rQ\right)\leq V_d\left({\rm conv}_rP\right).
\end{equation}
Furthermore, (\ref{Bezdek-10}) and (\ref{Bezdek-11}) imply 
\begin{equation}\label{Bezdek-31}
r_{in}^{\oo}\left(S_{r, r_0, d}\right)=r-\sqrt{r^2-r_0^2}\leq  r_{in}^{\oo}\left({\rm conv}_rQ\right),
\end{equation}
where $S_{r, r_0, d}$ is an $r$-spindle having $r_{cr}^{\oo}\left(S_{r, r_0, d}\right)=r_{0}$ in $\Ee^d$. Thus, it follows that
\begin{equation}\label{Bezdek-32} 
V_d(S_{r, r_0, d})=\int_{0}^{r_0}\sigma(x\Ss^{d-1}\cap  S_{r, r_0, d})dx \ {\rm and}\   V_d({\rm conv}_rQ)=\int_{0}^{r_0}\sigma(x\Ss^{d-1}\cap {\rm conv}_rQ)dx,
\end{equation} 
where $\sigma$ denotes the proper spherical Lebesgue measure on $x\Ss^{d-1}$. Clearly, (\ref{Bezdek-31}) yields that
\begin{equation}\label{Bezdek-33}
\sigma(x\Ss^{d-1}\cap  S_{r, r_0, d})= \sigma(x\Ss^{d-1}\cap {\rm conv}_rQ)
\end{equation}
holds for all $0\leq x\leq r-\sqrt{r^2-r_0^2}$. Finally, let $r-\sqrt{r^2-r_0^2}<x\leq r_0$. On the one hand, notice that the subset $x\Ss^{d-1}\cap  S_{r, r_0, d}$ of $x\Ss^{d-1}$ is the union of a pair of antipodal spherical caps of angular radius $0\leq \epsilon<\frac{\pi}{2}$. On the other hand, the subset $x\Ss^{d-1}\cap {\rm conv}_rQ$ of $x\Ss^{d-1}$ contains the union of $l+1$ spherical caps of angular radius $\epsilon$ centered at the points $\frac{x}{r_0}\p_i$, $1\leq i\leq l+1$, which do not lie on an open hemisphere of $x\Ss^{d-1}$ because $\oo\in{\rm relint}({\rm conv}(Q))$. Hence, Lemma~\ref{spherical-neighbourhood-sets} implies that
\begin{equation}\label{Bezdek-34}
\sigma(x\Ss^{d-1}\cap  S_{r, r_0, d})\leq \sigma(x\Ss^{d-1}\cap {\rm conv}_rQ)
\end{equation}
holds for all $r-\sqrt{r^2-r_0^2}<x\leq r_0$. Thus, (\ref{Bezdek-30}), (\ref{Bezdek-32}), (\ref{Bezdek-33}), and (\ref{Bezdek-34}) yield (\ref{Bezdek-4}) in a straightforward way. This completes the proof of Theorem~\ref{min-volume-spindle}.

\section{Proof of Corollary~\ref{lower-bound-for-intrinsic-volume}}


On the one hand, the extended isoperimetric inequality (see for example, (1.1) in \cite{PaPi}) yields
\begin{equation}\label{Bezdek-39}
\left(\frac{V_d({\rm conv}_rP)}{V_d(\B^d[\oo, 1])}\right)^{\frac{1}{d}}\leq \left(\frac{V_k({\rm conv}_rP)}{V_k(\B^d[\oo, 1])}\right)^{\frac{1}{k}},
\end{equation}
where $1\leq k<d$. On the other hand, recall (\cite{Sc}) that
\begin{equation}\label{Bezdek-40}
V_k(\B^d[\oo, 1])=\frac{\binom{d}{k}\omega_d}{\omega_{d-k}}
\end{equation}
holds for all $1\leq k\leq d$. Hence, (\ref{Bezdek-39}), (\ref{Bezdek-40}), and Theorem~\ref{min-volume-spindle}  imply
\begin{equation}\label{Bezdek-41}
V_k({\rm conv}_rP)\geq V_k(\B^d[\oo ,1])\frac{1}{{\omega}_d^{\frac{k}{d}}}\bigg[V_d({\rm conv}_rP)\bigg]^{\frac{k}{d}}\geq \frac{\binom{d}{k}\omega_d^{1-\frac{k}{d}}}{\omega_{d-k}}\left(V_d(S_{r, r_0, d})\right)^{\frac{k}{d}}
\end{equation}
for all $1\leq k<d$, finishing the proof of (\ref{Bezdek-6}).

Now, we turn to the proof of (\ref{Bezdek-5}). Proposition 2.5 of \cite{BeNa} and (\ref{Bezdek-9a}) imply
\begin{equation}\label{Bezdek-35}
\B^2[\oo, r]={\rm conv}_rP-\left({\rm conv}_rP\right)^r={\rm conv}_rP-P^r,
\end{equation}
from which one obtains 
\begin{equation}\label{Bezdek-36}
V_1(\B^2[\oo, r])=V_1({\rm conv}_rP)+V_1(P^r).
\end{equation}
Using Remark~\ref{planar case} and (\ref{Bezdek-36}) we get that
\begin{equation}\label{Bezdek-37}
V_1(\B^2[\oo, r])-V_1\left(L_{r, r-r_0, 2}\right)\leq V_1(\B^2[\oo, r])-V_1(P^r)=V_1({\rm conv}_rP).
\end{equation}
Next, notice that (\ref{Bezdek-36}) for ${\rm conv}_rP=S_{r, r_0, 2}$ and $P^r=L_{r, r-r_0, 2}$ yields
\begin{equation}\label{Bezdek-38}
V_1\left(S_{r, r_0, 2}\right)=V_1(\B^2[\oo, r])-V_1\left(L_{r, r-r_0, 2}\right).
\end{equation}
Finally, (\ref{Bezdek-37}) and (\ref{Bezdek-38}) imply (\ref{Bezdek-5}) in a straightforward way. This completes the proof of Corollary~\ref{lower-bound-for-intrinsic-volume}.

\bigskip


\noindent K\'aroly Bezdek \\
\small{Department of Mathematics and Statistics, University of Calgary, Canada}\\
\small{Department of Mathematics, University of Pannonia, Veszpr\'em, Hungary\\
\small{E-mail: \texttt{bezdek@math.ucalgary.ca}}

\end{document}